\documentclass[a4paper,11pt]{amsart}
\usepackage{graphicx}
\usepackage{mathptmx}
\usepackage{amsmath}
\usepackage{amssymb}
\usepackage{enumitem}
\usepackage{xcolor}
\usepackage{xparse}
\NewDocumentCommand{\eulerian}{omm}
 {%
  \genfrac<>{0pt}{}{#2}{#3}%
  \IfValueT{#1}{_{\!#1}}%
 }

 \newmuskip\pFqmuskip

\newcommand*\pFq[6][8]{%
  \begingroup 
  \pFqmuskip=#1mu\relax
  \mathchardef\normalcomma=\mathcode`,
  \mathcode`\,=\string"8000
  \begingroup\lccode`\~=`\,
  \lowercase{\endgroup\let~}\pFqcomma
  {}_{#2}F_{#3}{\left(\genfrac..{0pt}{}{#4}{#5}\bigg|#6\right)}%
  \endgroup
}
\newcommand{\pFqcomma}{{\normalcomma}\mskip\pFqmuskip}

\newtheorem{theorem}{Theorem}
\newtheorem{lemma}[theorem]{Lemma}

\begin{document}

\title[A new approach to fully degenerate Bernoulli numbers and polynomials]{A new Approach to fully degenerate Bernoulli numbers and polynomials}

\author{Taekyun  Kim}
\address{Department of Mathematics, Kwangwoon University, Seoul 139-701, Republic of Korea}
\email{tkkim@kw.ac.kr}

\author{DAE SAN KIM}
\address{Department of Mathematics, Sogang University, Seoul 121-742, Republic of Korea}
\email{dskim@sogang.ac.kr}

\subjclass[2010]{11B68; 11B73; 11B83}
\keywords{fully degenerate Bernoulli polynomials; degenerate Euler polynomials; degenerate $r$-Stirling numbers of the second kind}

\maketitle

\begin{abstract}
In this paper, we consider the doubly indexed sequence $a_{\lambda}^{(r)}(n,m),\,\,(n, m \ge 0)$, defined by a recurrence relation and an initial sequence $a_{\lambda}^{(r)}(0,m),\,\,(m \ge 0)$. We derive with the help of some differential operator an explicit expression for $a_{\lambda}^{(r)}(n,0)$, in term of the degenerate $r$-Stirling numbers of the second and the initial sequence. We observe that $a_{\lambda}^{(r)}(n,0)=\beta_{n,\lambda}(r)$, for $a_{\lambda}^{(r)}(0,m)=\frac{1}{m+1}$, and $a_{\lambda}^{(r)}(n,0)=\mathcal{E}_{n,\lambda}(r)$, for $a_{\lambda}^{(r)}(0,m)=\big(\frac{1}{2}\big)^{m}$. Here $\beta_{n,\lambda}(x)$ and $\mathcal{E}_{n,\lambda}(x)$ are the fully degenerate Bernoulli polynomials and the degenerate Euler polynomials, respectively.
\end{abstract}

\section{Introduction}
In recent years, we have witnessed that some mathematicians have explored various degenerate versions of many special polynomials and numbers by using various tools, which was initiated by Carlitz when he studied degenerate versions of some special polynomials and numbers, namely the degenerate Bernoulli and Euler polynomials and numbers (see [2]).  \par
The $r$-Stirling number of the second kind ${n \brace k}_{r}$ counts the number of partitions of the set $[n]=\left\{1,2,\dots,n \right\}$ into $k$ non-empty disjoint subsets in such a way that the numbers $1,2,\dots,r$ are in distinct subsets. The degenerate $r$-Stirling numbers of the second kind ${n \brace k}_{r,\lambda}$ are a degenerate version of the $r$-Stirling numbers of the second kind ${n \brace k}_{r}$. They can be viewed also as natural extensions of the degenerate Stirling numbers of the second kind ${n \brace k}_{\lambda}$, which were introduced earlier (see [11,12]). \par
The aim of this paper is to derive an explicit expression for the $n$-th generating function $g_{n}(t,\lambda)=\sum_{m=0}^{\infty}a_{\lambda}(n,m)t^{m}$, from the initial generating function $g_{0}(t,\lambda)=\sum_{m=0}^{\infty}a_{\lambda}(0,m)t^{m}$, and the recurrence relation given by \eqref{12}. If we choose $a_{\lambda}(0,k)=\frac{1}{k+1}$, then $a_{\lambda}(n,0)=\beta_{n,\lambda}$. Here $\beta_{n,\lambda}(x)$ are the fully degenrate Bernoulli polynomials and $\beta_{n,\lambda}=\beta_{n,\lambda}(0)$ are the fully degenerate Bernoulli numbers (see \eqref{10}). This is generalized to slightly more general recurrence relation in \eqref{29} with initial generating function $g_{0}^{(r)}(t,\lambda)=\sum_{m=0}^{\infty}a_{\lambda}^{(r)}(0,m)t^{m}$. Then we get $a_{\lambda}^{(r)}(n,0)=\beta_{n,\lambda}(r)$, for $a_{\lambda}^{(r)}(0,m)=\frac{1}{m+1}$, and $a_{\lambda}^{(r)}(n,0)=\mathcal{E}_{n,\lambda}(r)$, for $a_{\lambda}^{(r)}(0,m)=\big(\frac{1}{2}\big)^{m}$. Here $\mathcal{E}_{n,\lambda}(x)$ are the degenerate Euler polynomials. \par
In more detail, we show the following. For a given initial degenerate sequence $a_{\lambda}(0,m),\,\,(m=0,1,2,\dots)$, the doubly indexed degenerate sequence $a_{\lambda}(n,m),\,\,(n \ge 1,\,m=0,1,2,\dots)$ are defined by the recurrence relation in \eqref{12}. For each nonnegative integer $n$, let $g_{n}(t,\lambda)=\sum_{m=0}^{\infty}a_{\lambda}(n,m)t^{m}$ be the $n$-th generating function. Then we show that
\begin{align*}
g_{n}(t,\lambda)=\Big((t-1)\frac{d}{dt}\Big)_{n,\lambda}g_{0}(t,\lambda)=\sum_{k=0}^{n}S_{2,\lambda}(n,k)(t-1)^{k}\big(\frac{d}{dt}\big)^{k}g_{0}(t,\lambda),
\end{align*}
where $(x)_{n,\lambda}$ are the degenerate falling factorials, and $S_{2,\lambda}(n,k)$ are the degenerate Stirling numbers of the second kind (see \eqref{3}). Then, by letting $t=0$, we express $a_{\lambda}(n,0)$ in terms of the initial sequence $a_{\lambda}(0,k)$, namely $a_{\lambda}(n,0)=\sum_{k=0}^{n}S_{2,\lambda}(n,k)(-1)^{k}k!a_{\lambda}(0,k)$. For $a_{\lambda}(0,k)=\frac{1}{k+1}$, $a_{\lambda}(n,0)=\beta_{n,\lambda}$. This idea is generalized to the case of slightly more general recurrence relation in \eqref{29}, starting with the initial sequence $a_{\lambda}^{(r)}(0,m)$. Indeed, by proceeding similarly to the previous case we obtain the expression
\begin{align*}
a_{\lambda}^{(r)}(n,0)=\sum_{m=0}^{n}{n+r \brace m+r}_{\lambda}(-1)^{m}m!a_{\lambda}^{(r)}(0,m),
\end{align*}
where $r$ is a nonnegative integer. For $a_{\lambda}^{(r)}(0,m)=\frac{1}{m+1}$, $a_{\lambda}^{(r)}(n,0)=\beta_{n,\lambda}(r)$. In addition, we express the sum $\sum_{n=0}^{\infty}a_{\lambda}^{(r)}(n,0)\frac{t^n}{n!}$ in term of the initial generating function $F(t)=g_{0}^{(r)}(t,\lambda)$. Observe here that the sum is over the first argument. Indeed, we have
\begin{align*}
e_{\lambda}^{r}(t)F(1-e_{\lambda}(t))=\sum_{n=0}^{\infty}a_{\lambda}^{(r)}(n,0)\frac{t^n}{n!}.
\end{align*}
Here $a_{\lambda}^{(r)}(n,0)=\mathcal{E}_{n,\lambda}(r)$, for $a_{\lambda}^{(r)}(0,m)=\big(\frac{1}{2}\big)^{m}$.

For any $\lambda\in\mathbb{R}$, the degenerate exponential functions are defined by
\begin{equation}
e_{\lambda}^{x}(t)=(1+\lambda t)^{\frac{x}{\lambda}},\quad (\mathrm{see}\ [5,8-12]).\label{1}
\end{equation}
When $x=1$, we let $e_{\lambda}(t)=e_{\lambda}^{1}(t)$. Note that $\displaystyle \lim_{\lambda\rightarrow 0}e_{\lambda}^{x}(t)=e^{xt}\displaystyle$. \par
The degenerate falling factorials are defined by
\begin{displaymath}
(x)_{0,\lambda}=1,\quad (x)_{n,\lambda}=x(x-\lambda)(x-2\lambda)\cdots\big(x-(n-1)\lambda\big),\quad (n\ge 1).
\end{displaymath}
In [5], the degenerate Stirling numbers of the first kind are defined by
\begin{equation}
(x)_{n}=\sum_{k=0}^{n}S_{1,\lambda}(n,k)(x)_{k,\lambda},\quad (n\ge 0),\label{2}	
\end{equation}
where $(x)_{0}=1,\ (x)_{n}=x(x-1)\cdots(x-n+1),\ (n\ge 1)$. \par
As the inversion formula of \eqref{2}, the degenerate Stirling numbers of the second kind are defined by
\begin{equation}
(x)_{n,\lambda}=\sum_{k=0}^{n}S_{2,\lambda}(n,k)(x)_{k},\quad (n\ge 0),\quad (\mathrm{see}\ [5]).\label{3}
\end{equation}
Note that $\displaystyle\lim_{n\rightarrow \infty}S_{1,\lambda}(n,k)=S_{1}(n,k),\ \lim_{\lambda\rightarrow 0}S_{2,\lambda}(n,k)=S_{2}(n,k)\displaystyle$, where $S_{1}(n,k)$ and $S_{2}(n,k)$ are respectively the Stirling numbers of the first kind and the Stirling numbers of the second kind defined by
\begin{equation}
(x)_{n}=\sum_{k=0}^{n}S_{1}(n,k)x^{k},\quad x^{n}=\sum_{k=0}^{n}S_{2}(n,k)(x)_{k},\quad (n\ge 0),\quad (\mathrm{see}\ [9,13-17]).\label{4}
\end{equation}
For $r\in\mathbb{Z}$ with $r\ge 0$, the degenerate $r$-Stirling numbers of the second kind are defined by
\begin{equation}
(x+r)_{n,\lambda}=\sum_{k=0}^{n}{n+r \brace k+r}_{\lambda}(x)_{k},\quad (n\ge 0),\quad (\mathrm{see}\ [11,12]).\label{5}	
\end{equation}
From \eqref{5}, we note that
\begin{equation}
\frac{1}{k!}e_{\lambda}^{r}(t)\big(e_{\lambda}(t)-1\big)^{k}=\sum_{n=k}^{\infty}{n+r \brace k+r}_{\lambda}\frac{t^{n}}{n!},\quad (\mathrm{see}\ [11,12]). \label{6}
\end{equation}
By \eqref{5} and \eqref{6}, we easily get
\begin{equation}
{n+r \brace k+r}_{\lambda}=\sum_{m=0}^{n}\binom{n}{m}S_{2,\lambda}(m,k)(r)_{n-m,\lambda}\label{7}.
\end{equation}
For $r=0$, we note that ${n \brace k}_{\lambda}=S_{2,\lambda}(n,k)$. \\
In [2], Carlitz introduced the degenerate Bernoulli polynomials defined by
\begin{equation}
\frac{t}{e_{\lambda}(t)-1}e_{\lambda}^{x}(t)=\sum_{n=0}^{\infty}B_{n}(x|\lambda)\frac{t^{n}}{n!}.\label{8}
\end{equation}
Note that $\displaystyle\lim_{\lambda\rightarrow 0} B_{n}(x|\lambda) =B_{n}(x)\displaystyle$, where $B_{n}(x)$ are the Bernoulli polynomials defined by
\begin{displaymath}
\frac{t}{e^{t}-1}e^{x t}=\sum_{n=0}^{\infty}B_{n}(x)\frac{t^{n}}{n!},\quad (\mathrm{see}\ [1,3,14]).
\end{displaymath}
He also defined the degenerate Euler polynomials given by
\begin{equation}
\frac{2}{e_{\lambda}(t)+1}e_{\lambda}^{x}(t)=\sum_{n=0}^{\infty}	\mathcal{E}_{n,\lambda}(x)\frac{t^{n}}{n!},\quad (\mathrm{see}\ [2]). \label{9}
\end{equation}
When $x=0$, $\mathcal{E}_{n,\lambda}=\mathcal{E}_{n,\lambda}(0)$ are called the degenerate Euler numbers. \par
The fully degenerate Bernoulli polynomials arise from a $p$-adic bosonic integral over $\mathbb{Z}_{p}$ and are given by
\begin{equation}
\frac{\log(1+\lambda t)}{\lambda(e_{\lambda}(t)-1)}e_{\lambda}^{x}(t)=\sum_{n=0}^{\infty}\beta_{n,\lambda}(x)\frac{t^{n}}{n!},\quad (\mathrm{see}\ [9]).\label{10}
\end{equation}
When $x=0$, $\beta_{n,\lambda}=\beta_{n,\lambda}(0)$ are called the fully degenerate Bernoulli numbers. \par
For $n\ge 0$, we note that
\begin{equation}
\beta_{n,\lambda}(r)=\sum_{k=0}^{n}{n+r \brace k+r}_{\lambda}(-1)^{k}\frac{k!}{k+1},\quad (r\ge 0),\quad (\mathrm{see}\ [9]). \label{11}
\end{equation}

\section{A new approach to fully degenerate Bernoulli numbers and polynomials}
In this section, we prove the main results of this paper.
Assume that $\lambda$ is a fixed real number. For a given initial degenerate sequence $a_{\lambda}(0,m)$ $(m=0,1,2,\dots)$, we define the doubly indexed degenerate sequence $a_{\lambda}(n,m),\ (n \ge 1,\,\, m \ge 0)$, which are given by
\begin{equation}
a_{\lambda}(n,m)=ma_{\lambda}(n-1,m)-(m+1)a_{\lambda}(n-1,m+1)-(n-1)\lambda a_{\lambda}(n-1,m).\label{12}
\end{equation}
For any nonnegative integer $n$, let $g_{n}(t,\lambda)$ be the generating function of $a_{\lambda}(n,m),\,\,(m \ge 0)$, given by
\begin{equation}
g_{n}(t,\lambda)=\sum_{m=0}^{\infty}a_{\lambda}(n,m)t^{m}. \label{13}
\end{equation}
From \eqref{12} and \eqref{13}, we note that
\begin{align}
&g_{n}(t,\lambda)=\sum_{m=0}^{\infty}a_{\lambda}(n,m)t^{m}\label{14} \\
&=\sum_{m=0}^{\infty}\Big\{ma_{\lambda}(n-1,m)-(m+1)a_{\lambda}(n-1,m+1)-(n-1)\lambda	a_{\lambda}(n-1,m)\Big\}t^{m}\nonumber\\
&=\sum_{m=1}^{\infty}ma_{\lambda}(n-1,m)t^{m}-\sum_{m=0}^{\infty}(m+1)a_{\lambda}(n-1,m+1)t^{m}-(n-1)\lambda\sum_{m=0}^{\infty}a_{\lambda}(n-1,m)t^{n}\nonumber \\
&=t\sum_{m=0}^{\infty}(m+1)a_{\lambda}(n-1,m+1)t^{m}-\sum_{m=0}^{\infty}(m+1)a_{\lambda}(n-1,m+1)t^{m}\nonumber\\
&\quad -(n-1)\lambda\sum_{m=0}^{\infty}a_{\lambda}(n-1,m)t^{m}\nonumber \\
&=(t-1)\sum_{m=0}^{\infty}(m+1)a_{\lambda}(n-1,m+1)t^{m}-(n-1)\lambda\sum_{m=0}^{\infty}a_{\lambda}(n-1,m)t^{m}\nonumber \\
&=(t-1)\sum_{m=1}^{\infty}ma_{\lambda}(n-1,m)t^{m-1}-(n-1)\lambda\sum_{m=0}^{\infty}a_{\lambda}(n-1,m)t^{m}\nonumber \\
&=(t-1)\frac{d}{dt}g_{n-1}(t,\lambda)-(n-1)\lambda g_{n-1}(t,\lambda) \nonumber \\
&=\Big((t-1)\frac{d}{dt}-(n-1)\lambda
\Big)g_{n-1}(t,\lambda)\nonumber.
\end{align}
Thus, by \eqref{14}, we get
\begin{equation}
\begin{aligned}
	g_{n}(t,\lambda)&=\Big((t-1)\frac{d}{dt}-(n-1)\lambda
\Big)g_{n-1}(t,\lambda)\\
& =\Big((t-1)\frac{d}{dt}-(n-1)\lambda
\Big)\Big((t-1)\frac{d}{dt}-(n-2)\lambda\Big)g_{n-2}(t,\lambda).
\end{aligned}\label{15}	
\end{equation}
Continuing this process, we have
\begin{equation}
g_{n}(t,\lambda)=\Big((t-1)\frac{d}{dt}\Big)_{n,\lambda}g_{0}(t,\lambda),\quad (n\in\mathbb{N}).\label{16}
\end{equation}
Here one has to observe that
\begin{align*}
\Big((t-1)\frac{d}{dt}-i\lambda \Big)\Big((t-1)\frac{d}{dt}-j\lambda \Big)=
\Big((t-1)\frac{d}{dt}-j\lambda \Big)\Big((t-1)\frac{d}{dt}-i\lambda \Big),
\end{align*}
for any distinct nonnegative integers $i, j$. Similar observations to this are needed for other results in below.
\begin{lemma}
For $k\in\mathbb{Z}$ with $k\ge 0$, we have
\begin{displaymath}
\Big((t-1)\frac{d}{dt}\Big)_{k,\lambda}=\sum_{m=0}^{k}S_{2,\lambda}(k,m)(t-1)^{m}\Big(\frac{d}{dt}\Big)^{m}.
\end{displaymath}
\end{lemma}
\begin{proof}
Let $f(t)=\sum_{n=0}^{\infty}a_{n}(t-1)^{n}$.
Then we have
\begin{align}
&\Big((t-1)\frac{d}{dt}\Big)_{k,\lambda}f(t)=\sum_{n=0}^{\infty}a_{n}(n)_{k,\lambda}(t-1)^{n}\label{17}\\
&=\sum_{n=0}^{\infty}a_{n}\Big(\sum_{m=0}^{k}S_{2,\lambda}(k,m)(n)_{m}\Big)(t-1)^{n}=\sum_{m=0}^{k}S_{2,\lambda}(k,m)\sum_{n=0}^{\infty}a_{n}(n)_{m}(t-1)^{n}\nonumber \\
&=\Big(\sum_{m=0}^{k}S_{2,\lambda}(k,m)(t-1)^{m}\Big(\frac{d}{dt}\Big)^{m}\Big)f(t).\nonumber
\end{align}
By \eqref{17}, we get
\begin{equation}
\Big((t-1)\frac{d}{dt}\Big)_{k,\lambda}=\sum_{m=0}^{k}S_{2,\lambda}(k,m)(t-1)^{m}\Big(\frac{d}{dt}\Big)^{m}. \label{18}
\end{equation}
\end{proof}
 From \eqref{16} and Lemma 1. we note that
 \begin{equation}
 \begin{aligned}
 g_{n}(t,\lambda)&=\Big((t-1)\frac{d}{dt}\Big)_{n,\lambda}g_{0}(t,\lambda) \\
 &=\sum_{k=0}^{n}S_{2,\lambda}(n,k)(t-1)^{k}\Big(\frac{d}{dt}\Big)^{k}g_{0}(t,\lambda),
 \end{aligned}\label{19}
 \end{equation}
where $\displaystyle g_{n}(t,\lambda)=\sum_{m=0}^{\infty}a_{\lambda}(n,m)t^{m}\displaystyle$, \par
Let $t=0$ in \eqref{19}. Then we have
\begin{equation}
a_{\lambda}(n,0)=\sum_{k=0}^{n}S_{2,\lambda}(n,k)(-1)^{k}k!a_{\lambda}(0,k).\label{20}	
\end{equation}
Therefore, by \eqref{20}, we obtain the following theorem.
\begin{theorem}
For a given initial degenerate sequence $a_{\lambda}(0,m),\ (m=0,1,2,\dots)$, let the doubly indexed degenerate sequence $a_{\lambda}(n,m)$ be defined by the recurrence relation
\begin{displaymath}
a_{\lambda}(n,m)=ma_{\lambda}(n-1,m)-(m+1)a_{\lambda}(n-1,m+1)-(n-1)\lambda a_{\lambda}(n-1,m),
\end{displaymath}	
where $n$ is a positive integer. Then we have
\begin{displaymath}
a_{\lambda}(n,0)=\sum_{k=0}^{n}S_{2,\lambda}(n,k)(-1)^{k}k!a_{\lambda}(0,k).
\end{displaymath}
\end{theorem}
Let $a_{\lambda}(0,k)=\frac{1}{k+1}$ in Theorem 2. Then, from \eqref{11}, we see that
\begin{equation}
a_{\lambda}(n,0)=\sum_{k=0}^{n}S_{2,\lambda}(n,k)(-1)^{k}\frac{k!}{k+1}=\beta_{n,\lambda},\quad (n\ge 0).\label{21}
\end{equation}
For a given initial degenerate sequence $b_{\lambda}(0,m),\ (m=0,1,2,\dots)$, consider the doubly indexed degenerate sequence $b_{\lambda}(n,m),\ (n\ge 1,\,\, m \ge 0)$, which are defined by
\begin{equation}
b_{\lambda}(n,m)=(m+1)b_{\lambda}(n-1,m)-(m+1)b_{\lambda}(n-1,m+1)-(n-1)\lambda b_{\lambda}(n-1,m).\label{22}
\end{equation}
Let $g_{n}^{*}(t,\lambda)$ be the generating function of $b_{\lambda}(n,m),\,\,(m \ge 0)$, which is given by
\begin{equation}
g_{n}^{*}(t,\lambda)=\sum_{m=0}^{\infty}b_{\lambda}(n,m)t^{m}.\label{23}	
\end{equation}
Then, by \eqref{22} and \eqref{23}, we get
\begin{align}
&g_{n}^{*}(t,\lambda)=\sum_{m=0}^{\infty}b_{\lambda}(n,m)t^{m}\label{24}	\\
&=\sum_{m=0}^{\infty}mb_{\lambda}(n-1,m)t^{m}-\sum_{m=0}^{\infty}(m+1)b_{\lambda}(n-1,m+1)t^{m}+\big(1-(n-1)\lambda\big)\sum_{m=0}b_{\lambda}(n-1,m)t^{m}\nonumber\\
&=(t-1)\sum_{m=0}^{\infty}(m+1)b_{\lambda}(n-1,m+1)t^{m}+\big(1-(n-1)\lambda\big)\sum_{m=0}^{\infty}b_{\lambda}(n-1,m)t^{m}\nonumber \\
&=(t-1)\sum_{m=1}^{\infty}mb_{\lambda}(n-1,m)t^{m-1}+\big(1-(n-1)\lambda\big)\sum_{m=0}^{\infty}b_{\lambda}(n-1,m)t^{m}\nonumber \\
&=(t-1)\frac{d}{dt}g_{n-1}^{*}(t,\lambda)+\big(1-(n-1)\lambda\big)g_{n-1}^{*}(t,\lambda)=\Big((t-1)\frac{d}{dt}+1-(n-1)\lambda\Big)g_{n-1}^{*}(t,\lambda)\nonumber.
\end{align}
Thus, by \eqref{24}, we get
\begin{align*}
g_{n}^{*}(t,\lambda)&=\Big((t-1)\frac{d}{dt}+1-(n-1)\lambda\Big)g_{n-1}^{*}(t,\lambda)\\
&=\Big((t-1)\frac{d}{dt}+1-(n-1)\lambda\Big)\Big((t-1)\frac{d}{dt}+1-(n-2)\lambda\Big)g_{n-2}^{*}(t,\lambda).
\end{align*}
Continuing this process, we have
\begin{equation}
g_{n}^{*}(t,\lambda)=\Big(1+(t-1)\frac{d}{dt}\Big)_{n,\lambda}g_{0}^{*}(t,\lambda), 	\label{25}
\end{equation}
where $\displaystyle g_{n}^{*}(t,\lambda)=\sum_{m=0}^{\infty}b_{\lambda}(n,m)t^{m}\displaystyle$.
\par
From \eqref{18} and \eqref{25}, we note that
\begin{align}
g_{n}^{*}(t,\lambda)&=\sum_{k=0}^{n}\binom{n}{k}(1)_{n-k,\lambda}\Big((t-1)\frac{d}{dt}\Big)_{k,\lambda}g_{0}^{*}(t,\lambda)\label{26}\\
&=\sum_{k=0}^{n}\binom{n}{k}(1)_{n-k,\lambda}\sum_{m=0}^{k}S_{2,\lambda}(k,m)(t-1)^{m}\Big(\frac{d}{dt}\Big)^{m}g_{0}^{*}(t,\lambda)\nonumber \\
&=\sum_{m=0}^{n}(t-1)^{m}\Big(\frac{d}{dt}\Big)^{m}\sum_{k=m}^{n}\binom{n}{k}(1)_{n-k,\lambda}S_{2,\lambda}(k,m)g_{0}^{*}(t,\lambda).\nonumber
\end{align}
 By \eqref{7} and \eqref{26}, we get
 \begin{equation}
 g_{n}^{*}(t,\lambda)=\sum_{m=0}^{n}(t-1)^{m}{n+1 \brace m+1}_{\lambda}\bigg(\frac{d}{dt}\bigg)^{m}g_{0}^{*}(t,\lambda), \label{27}	
 \end{equation}
where $\displaystyle g_{n}^{*}(t,\lambda)=\sum_{k=0}^{\infty}b_{\lambda}(n,k)t^{k}\displaystyle$. \par
Let $t=0$ in \eqref{27}. Then we have
\begin{equation}
b_{\lambda}(n,0)=\sum_{m=0}^{n}(-1)^{m}{n+1 \brace m+1}_{\lambda}m!b_{\lambda}(0,m).\label{28}
\end{equation}
Therefore, by \eqref{28}, we obtain the following theorem.
\begin{theorem}
For a given initial degenerate sequence $b_{\lambda}(0,m),\ (m=0,1,2,\dots)$, let the doubly indexed degenerate sequence $b_{\lambda}(n,m),\ (n \ge 1,\,\,m \ge 0)$, be defined by the recurrence relation
\begin{displaymath}
b_{\lambda}(n,m)=(m+1)b_{\lambda}(n-1,m)-(m+1)b_{\lambda}(n-1,m+1)-(n-1)\lambda b_{\lambda}(n-1,m).
\end{displaymath}
Then we have
\begin{displaymath}
b_{\lambda}(n,0)=\sum_{m=0}^{n}(-1)^{m}{n+1 \brace m+1}_{\lambda}m!b_{\lambda}(0,m).
\end{displaymath}
\end{theorem}
Let $b_{\lambda}(0,m)=\frac{1}{m+1}$ in Theorem 3. Then, by \eqref{11}, we have
\begin{align*}
b_{\lambda}(n,0)=\sum_{m=0}^{n}(-1)^{m}\frac{m!}{m+1}{n+1 \brace m+1}_{\lambda}=\beta_{n,\lambda}(1),\quad (n\ge 0).
\end{align*}
Let $r$ be a nonnegative integer. For a given initial degenerate sequence $a_{\lambda}^{(r)}(0,m),\ (m=0,1,2,\dots)$, consider the doubly indexed degenerate sequence $a_{\lambda}^{(r)}(n,m),\ (n\ge 1,\,\,m \ge 0)$, such that
\begin{equation}
a_{\lambda}^{(r)}(n,m)=(m+r)a_{\lambda}^{(r)}(n-1,m)-(m+1)a_{\lambda}^{(r)}(n-1,m+1)-(n-1)\lambda a_{\lambda}^{(r)}(n-1,m).\label{29}
\end{equation}
Let $g_{n}^{(r)}(t,\lambda)$ be the generating function of $a_{\lambda}^{(r)}(n,m),\,\,(m \ge 0)$, which is given by
\begin{equation}
g_{n}^{(r)}(t,\lambda)=\sum_{m=0}^{\infty}a_{\lambda}^{(r)}(n,m)t^{m}.\label{30}	
\end{equation}
By the same method as in \eqref{14} and \eqref{27}, we get
\begin{align}
g_{n}^{(r)}(t,\lambda)&=\Big(r+(t-1)\frac{d}{dt}\Big)_{n,\lambda}g_{0}^{(r)}(t,\lambda) \label{31} \\
&=\sum_{k=0}^{n}\binom{n}{k}(r)_{n-k,\lambda}\Big((t-1)\frac{d}{dt}\Big)_{k,\lambda}g^{(r)}(t,\lambda)\nonumber \\
&=\sum_{k=0}^{n}\binom{n}{k}(r)_{n-k,\lambda}\sum_{m=0}^{k}S_{2,\lambda}(k,m)(t-1)^{m}\Big(\frac{d}{dt}\Big)^{m}g_{0}^{(r)}(t,\lambda)\nonumber \\
&=\sum_{m=0}^{n}(t-1)^{m}\Big(\frac{d}{dt}\Big)^{m}\Big(\sum_{k=m}^{n}\binom{n}{k}(r)_{n-k,\lambda}S_{2,\lambda}(k,m)\Big)g_{0}^{(r)}(t,\lambda)\nonumber \\
&=\sum_{m=0}^{n}(t-1)^{m}{n+r \brace m+r}_{\lambda}\Big(\frac{d}{dt}\Big)^{m}g_{0}^{(r)}(t,\lambda),\nonumber
\end{align}
where $\displaystyle g_{n}^{(r)}(t,\lambda)=\sum_{k=0}^{n}a_{\lambda}^{(r)}(n,k)t^{k}\displaystyle$. \par
Let $t=0$ in \eqref{31}. Then we have
\begin{equation}
a_{\lambda}^{(r)}(n,0)=\sum_{m=0}^{n}{n+r \brace m+r}_{\lambda}(-1)^{m}m!a_{\lambda}^{(r)}(0,m).\label{32}
\end{equation}
Therefore, by \eqref{32}, we obtain the following theorem.
\begin{theorem}
Let $r$ be a nonnegative integer. For a given initial degenerate sequence $a_{\lambda}^{(r)}(0,m),\ (m=0,1,2,\dots)$, let the doubly indexed degenerate sequence $a_{\lambda}^{(r)}(n,m),\ (n\ge 1, \,\, m \ge 0)$, be defined by the recurrence relation
\begin{displaymath}
a_{\lambda}^{(r)}(n,m)=(m+r)a_{\lambda}^{(r)}(n-1,m)-(m+1)a_{\lambda}^{(r)}(n-1,m+1)-(n-1)\lambda a_{\lambda}^{(r)}(n-1,m).
\end{displaymath}
Then we have
\begin{displaymath}
a_{\lambda}^{(r)}(n,0)=\sum_{m=0}^{n}{n+r \brace m+r}_{\lambda}(-1)^{m}m!a_{\lambda}^{(r)}(0,m).
\end{displaymath}	
\end{theorem}
Let $a_{\lambda}^{(r)}(0,m)=\frac{1}{m+1}$. Then, by \eqref{11}, we get
\begin{displaymath}
a_{\lambda}^{(r)}(n,0)=\sum_{m=0}^{n}{n+r \brace m+r}_{\lambda}(-1)^{m}\frac{m!}{m+1}=\beta_{n,\lambda}(r),\quad (n\ge 0).
\end{displaymath}
For a given initial degenerate sequence $a_{\lambda}^{(r)}(0,m),\ (m=0,1,2,\dots)$, for the ease of notation we let
\begin{equation}
F(t)=g_{0}^{(r)}(t,\lambda)=\sum_{m=0}^{\infty}a_{\lambda}^{(r)}(0,m)t^{m}.\label{33}	
\end{equation}
From \eqref{6}, Theorem 4 and \eqref{33}, we note that
\begin{align}
\sum_{n=0}^{\infty}a_{\lambda}^{(r)}(n,0)\frac{t^{n}}{n!}&=\sum_{n=0}^{\infty}\bigg(\sum_{m=0}^{n}{n+r \brace m+r}_{\lambda}(-1)^{m}m!a_{\lambda}^{(r)}(0,m)\bigg)\frac{t^{m}}{m!}\label{34} \\
&=\sum_{m=0}^{\infty}(-1)^{m}m!	a_{\lambda}^{(r)}(0,m)\sum_{n=m}^{\infty}{n+r \brace m+r}_{\lambda}\frac{t^{n}}{n!}\nonumber \\
&=\sum_{m=0}^{\infty}(-1)^{m}m!	a_{\lambda}^{(r)}(0,m)e_{\lambda}^{r}(t)\frac{1}{m!}(e_{\lambda}(t)-1)^{m}\nonumber \\
&=e_{\lambda}^{r}(t)\sum_{m=0}^{\infty}a_{\lambda}^{(r)}(0,m)(1-e_{\lambda}(t))^{m}\nonumber \\
&=e_{\lambda}^{r}(t)F(1-e_{\lambda}(t))\nonumber.
\end{align}
Therefore, by \eqref{34}, we obtain the following theorem.
\begin{theorem}
Let $\displaystyle F(t)=g_{0}^{(r)}(t,\lambda)=\sum_{m=0}^{\infty}a_{\lambda}^{(r)}(0,m)t^{m}\displaystyle$. Then we have
\begin{displaymath}
e_{\lambda}^{r}(t)F(1-e_{\lambda}(t))=\sum_{n=0}^{\infty}a_{\lambda}^{(r)}(n,0)\frac{t^{n}}{n!},
\end{displaymath}	
where $r$ is a nonnegative integer.
\end{theorem}
Let $a_{\lambda}^{(r)}(0,m)=\big(\frac{1}{2}\big)^{m}$ in Theorem 5. Then we have
\begin{align}
\sum_{n=0}^{\infty}a_{\lambda}^{(r)}(n,0)\frac{t^{n}}{n!}&=e_{\lambda}^{r}(t)\sum_{m=0}^{\infty}a_{\lambda}(0,m)(1-e_{\lambda}(t))^{m}\label{35} \\
&=e_{\lambda}^{r}(t)\sum_{m=0}^{\infty}\bigg(-\frac{1}{2}\bigg)^{m}(e_{\lambda}(t)-1)^{m}\nonumber \\
&=e_{\lambda}^{r}(t)\frac{1}{1+\frac{1}{2}(e_{\lambda}(t)-1)}=\frac{2}{e_{\lambda}(t)+1}e_{\lambda}^{r}(t)\nonumber\\
&=\sum_{n=0}^{\infty}\mathcal{E}_{n,\lambda}(r)\frac{t^{n}}{n!}.\nonumber
\end{align}
Comparing the coefficients on both sides of \eqref{35}, we have
\begin{displaymath}
	a_{\lambda}^{(r)}(n,0)=\mathcal{E}_{n,\lambda}(r),\quad (n\ge 0),
\end{displaymath}
where $\mathcal{E}_{n,\lambda}(x)$ are the degenerate Euler polynomials.

\section{Conclusion}
In recent years, explorations for degerate versions of some special numbers and polynomials regained the interests of many mathematicians which began with Carlitz's work on the degenerate Bernoulli and Euler numbers. They have been done by using various tools like combinatorial methods, generating functions, differential equations, umbral calculus techniques, $p$-adic analysis, special functions, operator theory, probability theory, and analytic number theory. \par
For a given initial degenerate sequence $a_{\lambda}^{(r)}(0,m),\ (m=0,1,2,\dots)$, and the doubly indexed degenerate sequence $a_{\lambda}^{(r)}(n,m),\ (n\ge 1, \,\, m \ge 0)$, defined by the recurrence relation
\begin{displaymath}
a_{\lambda}^{(r)}(n,m)=(m+r)a_{\lambda}^{(r)}(n-1,m)-(m+1)a_{\lambda}^{(r)}(n-1,m+1)-(n-1)\lambda a_{\lambda}^{(r)}(n-1,m),
\end{displaymath}
it was shown by making use of a differential operator that
\begin{displaymath}
a_{\lambda}^{(r)}(n,0)=\sum_{m=0}^{n}{n+r \brace m+r}_{\lambda}(-1)^{m}m!a_{\lambda}^{(r)}(0,m).
\end{displaymath}	
Then we noted that
$a_{\lambda}^{(r)}(n,0)=\beta_{n,\lambda}(r)$, for $a_{\lambda}^{(r)}(0,m)=\frac{1}{m+1}$, and $a_{\lambda}^{(r)}(n,0)=\mathcal{E}_{n,\lambda}(r)$, for $a_{\lambda}^{(r)}(0,m)=\big(\frac{1}{2}\big)^{m}$. \par
It is one of our future projects to continue to explore many degenerate special numbers and polynomials with the help of aforementioned tools.

\end{document}